\documentclass[12pt,a4paper]{amsart}
\usepackage[top=35mm, bottom=30mm, left=30mm, right=30mm]{geometry}
\usepackage{mathptmx}
\usepackage{mathrsfs}

\usepackage{verbatim}
\usepackage{url}
\usepackage[all]{xy}
\usepackage{xcolor}
\usepackage[colorlinks=true,citecolor=blue]{hyperref}

\usepackage{amsmath,amsthm}
\usepackage{amssymb}

\usepackage{enumerate}

\usepackage{graphicx}

\newtheorem{thm}{Theorem}[section]

\newtheorem{lem}[thm]{Lemma}

\theoremstyle{definition}
\newtheorem{defin}[thm]{Definition}

\numberwithin{equation}{section}

\begin{document}
\title[The structures of pointwise recurrent quasi-graph maps]{The structures of pointwise recurrent quasi-graph maps}

\author[Z.Q.~Yu]{Ziqi Yu}
\address{Soochow University, Suzhou, Jiangsu 215006, China}
\email{20224007003@stu.suda.edu.cn}

\author[S.H.~Wang]{Suhua Wang}
\address{Suzhou Vocational University, Suzhou, Jiangsu 215104, China}
\email{wangsuhuajust@163.com}

\author[E.H.~Shi]{Enhui Shi}
\address{Soochow University, Suzhou, Jiangsu 215006, China}
\email{ehshi@suda.edu.cn}

\begin{abstract}
	We show that a continuous map $f$ from a quasi-graph $G$ to itself is pointwise recurrent if and only if one of the following two statements holds:
	(1) $X$ is a simple closed curve and $f$ is  topologically conjugate to an irrational rotation on the unit circle $\mathbb S^1$; (2) $f$ is a perodic homeomorphism.

\end{abstract}

\maketitle

\date{\today}

\section{Introduction}

It is well known that there always exist recurrent points for any continuous map $f$ on a compact metric space $X$ by Birkhoff recurrence theorem
and these recurrent points are large from the point of view of measures by Poincar\'e recurrence theorem.
These make  recurrence to be an important research subject in the theory of dynamical systems.  For a single point $x$ in $X$, there are several
natural classes of recurrence: periodic points, minimal points (or almost periodic points), distal points, recurrent points, etc.. It is interesting  to study the structure of $f$ if all points of $X$ possess the same kind of recurrence. There have been an intensively studied around this topic.
\medskip

  A classical result due to Ma\~n\'e says
that every minimal expansive homeomorphism $f$ is topologically conjugate to a subshift of some symbol system (\cite{Mane}),
which was extended to the case in which $f$ is both pointwise recurrent and expansive by Shi, Xu, and Yu very recently (\cite{SXY}).
Mai and Ye determined the structure of pointwise recurrent  maps having the pseudo orbit tracing property (\cite{MY}).
It is well known that distal systems are semisimple, that is the whole phase space is the disjoint union of minimal sets (see e.g. \cite{Ao}).
 The structure of minimal distal systems was completely described by Furstenberg
in \cite{Fu}. It is also known that a distal expansive system is finite (see e.g. \cite[Prop. 2.7.1]{BS}). Montgomery proved that every pointwise periodic homeomorphism on
a connected manifold is periodic (\cite{Mo}).
\medskip

Beside the above results for abstract systems, there are many interesting results around the structure of $f$ defined on  spaces of dimension $\leq 2$.
 Oversteegen and Tymchatyn proved that recurrent homeomorphisms on the plane are periodic (\cite{OT}); Kolev and P\a'erou\a`eme  established similar results for recurrent homeomorphisms on compact surfaces with negative Euler characteristic (\cite{KP}).  It is known that if $f$ is a pointwise recurrent map on a dendrite $X$,
 then the family $\{f^n:n=0,1,2,\cdots\}$ is equicontinuous and every cut point of $X$ is periodic (see e.g. \cite[Chap. XII]{Why}). For a pointwise recurrent graph map $f$,
 Mai showed that either it is topologically conjugate to an irrational rotation on the circle or it is of finite order (\cite{mai}).  Katznelson and Weiss obtained that
 if $f$ is a homeomorphism on a $0$-dimensional space $X$ which is both pointwise positively recurrent and transitive then it is minimal (\cite {KW}).
A celebrated result due to Auslander, Glasner, and Weiss says that a $0$-dimensional distal minimal system is equicontinuous (\cite{AGW}).

\medskip

The study of the dynamics of graph maps can date back to the work of Blokh in 1980's (\cite{Bl2, Bl3, Bl4}). Since then, lots of literatures appeared in this area.
To extend this theory to more general settings, Mai and Shi introduced the notion of quasi-graph and studied the topological structures and mapping properties of
quasi-graphs (\cite{ms}). Roughly speaking, a quasi-graph is a continuum which is the union of a graph and finitely many inner rays (see Section 2 for the definition).
One may consult \cite{LOZ} for the discussions about the topological entropy of quasi-graph maps. The aim of the paper is to proceed the study of the dynamics of quasi-graph mappings, especially focusing on the recurrence.
\medskip

The following is the main theorem of the paper.

\begin{thm} \label{main thm}
	Let $X$ be a quasi-graph, and $f:X \rightarrow X$ be a continuous map. Then $f$ is pointwise-recurrent if and only if one of the following two statements holds:
	\begin{itemize}
		\item[(1)] $X$ is a circle and $f$ is a homeomorphism topologically conjugate to an irrational ratation of the unit circle $\mathbb S^1$;
		\item[(2)] $f$ is a perodic homeomorphism, i.e., there exists  $m \in \mathbb{N}$ such that $f^m=id_X$.
	\end{itemize}	
\end{thm}

The proof relies on the main theorem by Mai in \cite{mai} and the techniques development in \cite{ms}. The Key idea is to show the existence of
a large  subgraph which is invariant under some power $f^k$ of $f$, and then reduce the proof to the case of graph mappings.

\section{Preliminaries}

In this section, we will recall some notions and results around recurrence and quasi-graphs.

\subsection{Structures of quasi-graphs}

 Let $X$ be a compact metric space with metric $d$. For any $A\subset X$, we use $\overline{A}$ and ${\AA}$ to denote the closure and the interior of $A$ in $X$ respectively. For $x\in X$ and $\epsilon>0$,  $B(x,\epsilon):=\{y\in X: d(x,y)<\epsilon\}$ denotes the open ball centered at $x$ with radius $\epsilon$.
 \medskip

A {\it continuum} is a compact connected metric space. An {\it arc} is a continuum homeomorphic to the closed interval $[0,1]$ and a {\it circle} is a continuum homeomorphic to the unit circle $\mathbb S^1$. An {\it $n$-star with center $v$} is a continuum which is the union of $n$ arcs with one endpoint $v$ being as a common intersecting point. By a {\it graph}, we mean a continuum which can be written as the union of finitely many arcs any two of which are either disjoint or intersect only in one or both of their endpoints. Each of these arcs is called an {\it edge} of $G$, and each endpoint of an edge is called a {\it vertex}. One may refer to \cite{Na} for details about graphs. \medskip

 Let $X$ be a compact arcwise connected metric space and let $v\in X$. The \textit{valence} of $v$ in $X$, denoted  by ${\rm val}(v,X)$ or simply by ${\rm val}(v)$, is the number ${\rm max}\{n\in \mathbb{N}$: there exists an $n$-star with center $v$ in $X$$\}$ (${\rm val}(v)$ may be $\infty$) ; $v$ is called an \textit{endpoint} of $X$ if ${\rm val}(v)=1$; $v$ is called a \textit{branch point} if ${\rm val}(v)\geq 3$. We use the symbols ${\rm End}(X)$ and ${\rm Br}(X)$ to denote the endpoint set and the branch point set of $X$ respectively. \medskip

The following definition of quasi-graph  was introduced in \cite{ms}.

\begin{defin} \label{quasi-graph}
	A nondegenerate compact arcwise connected metric space $X$ is called a \textit{quasi-graph} if there exists $N \in \mathbb{N}$, such that $\overline{Y}-Y$ has at most $N$ arcwise connected components, for every arcwise connected subset $Y$ of $X$.
%Furthermore, if such $N$ is minimal, i.e., if there exists an arcwise connected subset $Y_0$ of $X$ such that $\overline{Y_0}-Y_0$ has exactly $N$ arcwise connected components, %then $N$ is called the \textit{separation degree} of $X$.
\end{defin}

From the definition, we see that the Warsaw circle and graphs are examples of quasi-graphs.  It is not difficult to verify that the endpoint set and the branch point set of a quasi-graph are finite (see \cite{ms} for the proofs).

%\begin{exa}
%Let $W=\{(x,{\rm sin}(1/x))\in \mathbb{R}^2:0<x\leq 1\}$, and $\overline{W}=W\cup E$ where $E=\{(0,y)\in \mathbb{R}^2:-1\leq y\leq 1\}$. Then $\overline{W}$ is called the {\it ${\rm sin}(1/x)$-continuum}. Let $Z$ be the union of three convex arcs in $\mathbb{R}^2$, one from $(0,-1)$ to $(0,-2)$, one from $(0,-2)$ to $(1,-2)$, and one from $(1,-2)$ to $(1, {\rm sin}1)$.  The {\it Warsaw circle} is any continuum homeomorphic to $\overline{W}\cup Z$ (see \cite[p. 5]{Na}).
%\end{exa}
 \medskip

To describe the structures of quasi-graphs, we need some definitions.	Let $X$ be a compact arcwise connected metric space. A continuous injection $\varphi:\mathbb{R}_+ \rightarrow X$ is called a \textit{quasi-arc}, and $\varphi(0)$ is called the \textit{endpoint} of $\varphi$;  the image $L:=\varphi(\mathbb{R}_+)$ is also called
a {\it quasi-arc}. Set $\omega(L)=\omega(\varphi)=\cap\{\overline{L[m,\infty)}:m\in \mathbb{N}\}$, where $L[m,\infty)=\varphi([m, \infty))$. The set $\omega(L)$ (resp. $\omega(\varphi$)) is said to be the \textit{$\omega$-limit set} of $L$ (resp. $\varphi$). If $\omega(L)$ contains more than one point, then we call $L$ or $\varphi$ an \textit{oscillatory quasi-arc}.

 \medskip

Let $X$ be a quasi-graph and let $x\in X$. For every $\epsilon>0$, denote by ${\rm St}(x,\epsilon)={\rm St}(x,\epsilon,X)$ the arcwise connected component of $B(x,\epsilon)$ containing $x$. The following theorem  explicitly describes the topological structure of a quasi-graph.

\medskip
\begin{thm}\cite[Theorem 2.24]{ms}  \label{2.24}
A continuum $X$ is a quasi-graph if and only if there are a graph $G$ and n pairwise disjoint oscillatory quasi-arcs $L_1, \cdots ,L_n$ in $X$, for some  $n\in \mathbb{Z_+}$, such that

\begin{itemize}
  \item[(1)] $X=G \cup (\bigcup_{i=1}^{n}L_i)$, and ${\rm End}(X) \cup (\bigcup  \{{\rm St}(x,\epsilon_0):x \in {\rm Br}(X) \}) \subset G$ for some $ \epsilon_0 >0 $,
  \item[(2)] $L_i \cap G=\{a_i\} $ for each $1 \leq i \leq n$, where $a_i $ is the endpoint of $ L_i$,
  \item[(3)] $\omega(L_i) \subset G \cup (\bigcup_{j=1}^{i-1}L_j)$ for each $ 1 \leq i \leq n$, and
  \item[(4)] if $\omega(L_i) \cap L_j \neq \emptyset$ for some $i,j \in \{1,\cdots,n\}$, then $L_j \subset \omega(L_i)$.

\end{itemize}

\end{thm}

\subsection{Recurrence and graph mappings}

By a {\it dynamical system} we mean a pair $(X,f)$ where $X$ is a compact metric space and $f:X\rightarrow X$ is a continuous map. We use $\mathbb{N}$ and $\mathbb{Z}_+$ to denote the sets of positive integers and non-negative integers respectively. For any $x\in X$, $\{f^n(x):n\in \mathbb{Z}_{+}\}$ is called the {\it orbit} of $x$ under $f$ and is denoted by $O(x,f)$. A point $x\in X$ is called a {\it periodic point} of $f$ if $f^n(x)=x$ for some $n\in \mathbb{N}$. Particularly, $x$ is called a {\it fixed point} of $f$ if $f(x)=x$. A point $x\in X$ is called a {\it recurrent point} of $f$ if for any neighborhood $U$ of $x$ there exists $i\in \mathbb{N}$ such that $f^i(x)\in U$; this is equivalent to
saying that there exists an increasing sequence $n_1<n_2<\cdots$ such that $f^{n_i}(x)\rightarrow x\ (i\rightarrow\infty)$. If every point $x\in X$ is recurrent,
then $f$ is called  {\it pointwise recurrent}. Let ${\rm Fix}(f)$, $P(f)$ and $R(f)$ denote the sets of fixed points, period points and recurrent points respectively.
\medskip

The following lemma can be seen in \cite[Chapter IV, Lemma 25]{BC}.

\begin{lem}\label{recurrent}
Let $X$ be a compact metric space. Then for any $m\in {\mathbb{N}}$, $R(f^m)=R(f)$.

\end{lem}

 The following theorem is due to Mai.

\begin{thm}\cite[Theorem 4.4]{mai}  \label{1.1}
Let $G$ be a connected graph, and $f:G\to G$ be a continuous map. Then $f$ is pointwise-recurrent if and only if one of the following two statements holds:

  (1) $G$ is a circle and $f$ is a homeomorphism topologically conjugate to an irrational rotation of the unit circle $\mathbb{S}^1$;

  (2) $f$ is a periodic homeomorphism.
\end{thm}

\subsection{Properties of quasi-graph mappings}
For a compact metric space $X$, let $C^0(X)$ be the set of all continuous maps from $X$ to itself.
If $X$ is a quasi-graph, we call $f\in C^0(X)$ a {\it quasi-graph map}.

\begin{lem}\cite[Corollary 3.2]{ms}  \label{3.2}
	Let $X$ be a quasi-graph. Suppose that $G$ is a graph in $X$ and $ f \in C^0(X)$. If $f(G)$ contains more than one point, then $f(G)$ is also a graph in $X$.
\end{lem}

%\begin{prop}\cite[Propsition3.4]{ms} \label{3.4}
%	Let $X$ be a quasi-graph. Suppose that $L$ is a $k$-order oscillatory quasi-arc for some $k\geq 0$ and $f\in C^0(X)$. Then $f(L)$ contains no oscillatory %quasi-arcs with order $> k$.
%\end{prop}

In \cite{ms}, in order to characterize the equivalent conditions of recurrent points of quasi-graph maps, Mai and Shi introduced the notion of arcwise connectivity limit points by modifying the notion of limit points.

\medskip
\begin{defin} \label{arc connect limit}
	Let $X$ be a metric space. A point $v$ in $X$ is said to be a \textit{arcwise connectivity limit point} of a point sequence $x_1$, $x_2$, $\cdots$ in $X$ if for every $\epsilon >0$ and every $m\in \mathbb{N}$ there exists a arcwise connected set $W_\epsilon$ such that $v\in W_\epsilon \subset B(v,\epsilon)$ and $ W_\epsilon \cap \{x_m,x_{m+1}, \cdots\} \neq \emptyset$.
\end{defin}

From the above definition we see that an arcwise connectivity limit point must be a limit point, and for a locally arcwise connected space $X$, these two notions are equivalent.

\medskip
\begin{thm}\cite[Theorem5.4]{ms} \label{5.4}
		Let $X$ be a quasi-graph and let $ f \in C^0(X)$. Then for every $v \in X $ the following three items are equivalent:
		
\begin{itemize}
 \item[(1)] $v \in R(f)$.
 \item[(2)] There is $\epsilon_0>0$ such that, for every $\epsilon \in (0,\epsilon_0]$, $B(v,\epsilon) \cap O(f(v),f)=St(v,\epsilon) \cap O(f(v),f) \neq \emptyset$.
 \item[(3)] $v$ is an arcwise connected limit point of the point sequence $v,f(v),f^2(v),\cdots.$

\end{itemize}

\end{thm}

%\medskip
%\begin{thm} \cite[theorem4.4]{mai} \label{4.4}
%		Let $G$ be a graph, and $f:G \rightarrow G$ be a continuous map. Then $f$ is pointwise-recurrent if and only if one of the following two statements %holds:
%	\begin{itemize}
%		\item[(1)] $G$ is a circle and $f$ is a homeomorphism topologically conjugate to an irrational ratation of the unit circle $S^1$;
%		\item[(2)] $f$ is a perodic homeomorphism.
%	\end{itemize}
%	
%	\end{thm}

\section{Existence of $f^k$-invariant large subgraphs}

Let $X$ be a compact metric space and $f\in C^0(X)$. Recall that a subset $A$ of $X$ is called $f$-invariant if $f(A)\subset A$.

\begin{lem} \label{invariant graph}
Let $X=G\cup (\cup_{i=1}^n L_i)$ be a quasi-graph, where $G$ and $L_1, \cdots, L_n$ are as in Theorem \ref{2.24}. Suppose that $f:X \rightarrow X$ is a pointwise-recurrent continuous map. Then there is a graph $G'\supset G$ in $X$ such that $G'$ is $f^k$-invariant for some $k\in \mathbb{N}$.
\end{lem}

\begin{proof}
The lemma obviously holds if $X$ is a graph. So, we assume that $X$ contains at least one oscillatory quasi-arc; that is $X=G\cup (\cup_{i=1}^n L_i)$ with $n\neq 0$.
For each $i$, let $a_i$ be the unique endpoint of the oscillatory quasi-arc $L_i$, and let $\varphi_i: \mathbb{R}_+ \to L_i$ be the corresponding continuous bijection.

\smallskip
\noindent{\it Claim }1.\quad There is some $k\in \mathbb{N}$ such that $f^k(G)\cap G\neq\emptyset$.

\smallskip
\noindent{\it Proof of Claim }1.\quad Take a point $x\in G-\{a_1,\cdots, a_n\}$. Then there exists $\epsilon_x >0$ such that ${\rm St}(x,\epsilon_x)\subset G$. Since $x\in R(f)$, then by Theorem \ref{5.4}, there exists $k\in \mathbb{N}$ such that $f^k(x)\in {\rm St}(x,\epsilon_x)\subset G$, which implies that $f^k(G)\cap G\neq\emptyset$.

\smallskip
Let $g=f^k$. By Lemma~\ref{recurrent}, $g$ is also  pointwise-recurrent. According to Claim 1, we  see that $g^i(G)\cap g^{i-1}(G)\neq\emptyset$ for each $i\in \mathbb{N}$. In addition, $g^i(G)$ is a graph in $X$ for each $i\in \mathbb{N}$ by Lemma~\ref{3.2}. Hence $G_m:={\cup}_{i=0}^{m}g^i(G)$ is a graph in $X$ and $G_{m-1}\subset G_m$ for each $m\in \mathbb{N}$.
\smallskip

There are two cases.

\smallskip
\noindent{\bf Case 1.}\quad There exists some $M\in \mathbb{Z}_+$ such that $G_{M+1}=G_M$.
 Then $G_m=G_M$ for each positive integer $m\geq M$. So $g(G_M)={\cup}_{i=0}^{M}g^{i+1}(G)\subset G_{M+1}=G_M$. So $G':=G_M$ is a $g$-invariant graph in $X$.
Then $G'$ meets the requirement.

\smallskip
\noindent{\bf Case 2.}\quad For each $m\in \mathbb{Z}_+$, $G_m \subsetneq G_{m+1}$. Notice that if  $\varphi:\mathbb{R}_+\rightarrow L\subset X$
is an oscillatory quasi-arc and $\widetilde G$ is a graph in $X$ which contains $\varphi(0)$, then $\widetilde G\cap L=\varphi([0, s])$ for some $s\geq 0$.
Thus for each $L_i$, there exists a sequence $0=s_0^{(i)}\leq s_1^{(i)} \leq s_2^{(i)}\leq \cdots$ such that $G_m\cap L_i=\varphi_i([0, s_m^{(i)}])$, where $\varphi_i$ is the associated bijection $\varphi_i:\mathbb{R}_+\rightarrow L_i$.
It follows that $G_m=G\cup(\cup_{i=1}^n \varphi_i([0, s_m^{(i)}]))$ for each $m$. So $G_m-G_{m-1}=\cup_{i=1}^n C_{m,i}$, where
\begin{equation}
C_{m,i}=
\begin{cases} \varphi_i((s_{m-1}^{(i)},s_m^{(i)}]), & {\rm if\ \ } s_{m-1}^{(i)}<s_m^{(i)}, \\ \emptyset, & {\rm if\ \ } s_{m-1}^{(i)}=s_m^{(i)}.
\end{cases}
\end{equation}

Let $A_m=\{x\in G: g^m(x)\in G_m-G_{m-1}\}$. Then $A_m$ is nonempty  for each $m\in \mathbb{N}$ by the assumption.
\smallskip

\noindent{\it Claim }2.\quad   $A_{m}\subset A_{m-1}$.

\smallskip
\noindent{\it Proof of Claim }2.\quad For any $x\in A_m$, since $g^m(x)\notin G_{m-1}$, $g^{m-1}(x)\notin G_{m-2}$. On the other hand, $g^{m-1}(x)\in G_{m-1}$ by the definition of $G_{m-1}$. Set $g^{m-1}(x)\in G_{m-1}-G_{m-2}$ and the claim holds.

\smallskip
From Claim 2, $\{\overline{A_m}:m\in \mathbb{N}\}$ becomes a decreasing sequence of nonempty closed subsets of $G$. Then $\cap_{m=1}^{\infty}\overline{A_m}\neq\emptyset$. Take a point $v\in \cap_{m=1}^{\infty}\overline{A_m}$. Noting that $\overline{A_m}\subset \{x\in G: g^m(x)\in \overline{G_m-G_{m-1}}\}$,   we have $g^m(v)\in \overline{G_m-G_{m-1}}=\overline{\cup_{i=1}^{n}C_{m,i}}=\cup_{i=1}^{n}\overline{C_{m,i}}$, where
\begin{equation}
\overline{C_{m,i}}=
\begin{cases} \varphi_i([s_{m-1}^{(i)},s_m^{(i)}]), & {\rm if\ \ } s_{m-1}^{(i)}<s_m^{(i)}, \\ \emptyset, & {\rm if\ \ } s_{m-1}^{(i)}=s_m^{(i)}.
\end{cases}
\end{equation}

\noindent{\it Claim }3.\quad  $v\notin R(g)$.

\smallskip
\noindent{\it Proof of Claim }3.\quad We discuss into  two cases:
\smallskip

\noindent{\bf Subcase 2.1.}\quad $v\in G-\{a_1,\cdots, a_n\}$. Then by Theorem~\ref{2.24}, there exists $\epsilon>0$ such that ${\rm St}(v,\epsilon)\subset G-\{a_1,\cdots, a_n\}$. Thus we have ${\rm St}(v,\epsilon)\cap \overline{G_m-G_{m-1}}=\emptyset$ for all $m\in \mathbb{N}$, which implies that ${\rm St}(v,\epsilon)\cap O(g(v),g)=\emptyset$. So $v\notin R(g)$ by Theorem~\ref{5.4}. This is a contradiction.

\smallskip
\noindent{\bf Subcase 2.2.}\quad $v=a_i$ for some $i\in \{1,\cdots,n\}$; say $i=1$.
If $\overline{C_{m,1}}=\emptyset$ for all $m\in \mathbb{N}$, then we can take a sufficient small $\epsilon>0$ such that ${\rm St}(v,\epsilon)\subset (G-\{a_1,\cdots,a_n\})\cup L_1$. Thus ${\rm St}(v,\epsilon)\cap \overline{G_m-G_{m-1}}=\emptyset$ for all $m\in \mathbb{N}$. So, ${\rm St}(v,\epsilon)\cap O(g(v),g)=\emptyset$, which implies that $v\notin R(g)$. This is a contradiction. Thus there is some $m_1\in \mathbb{N}$ such that $\overline{C_{m_1,1}}=\varphi_1([0,s_{m_{1}}^{(1)}])$ for some $s_{m_{1}}^{(1)}>0$. Take an $\epsilon>0$ with ${\rm St}(v,\epsilon)\subset (G-\{a_1,\cdots,a_n\})\cup \varphi_1([0,s_{m_1}^{(1)}))$. Then for  $m>m_1$, ${\rm St}(v,\epsilon)\cap \overline{G_m-G_{m-1}}=\emptyset$, which means ${\rm St}(v,\epsilon)\cap \{g^m(v):m>m_1\}=\emptyset$. This implies that $v\notin R(g)$ and leads to a contradiction.
\smallskip

From Claim 3,  Case 2 cannot occur. Thus we complete the proof.
\end{proof}

\section{Proof of the main theorem}
First we recall some notions and notations. Let $X$ be a quasi-graph and $\varphi:\mathbb{R}_+\rightarrow L\subset X$ be
a quasi-arc. If there are $s, t\in \mathbb{R}_+$ such that $x=\varphi(s)$ and $y=\varphi(t)$, then we use $L[x, y]$ to denote $\varphi([s, t])$.
A point sequence $x_1, x_2, x_3, \cdots$ in $X$ is said to be {\it cofinal with} $L$ if there exist $m\in \mathbb{N}$ and positive numbers $t_0<t_1<t_2<\cdots$ such that $x_{m+i}=\varphi(t_i)$ for all $i\in \mathbb{R}_+$ and ${\rm lim}_{i\to \infty}t_i=\infty$.
\smallskip

To prove the main theorem, we need the following lemma.

\begin{lem} \label{fix quasi-arc}
	Let $X$ be a quasi-graph, and $f:X \rightarrow X$ be a pointwise recurrent continuous map. Suppose that $L$ is an oscillatory quasi-arc in $X$ with a (unique) endpoint $a$. If $L \cap {\rm Br}(X)=\emptyset$ and $a\in {\rm Fix}(f)$, then there exists some $k\in \mathbb{N}$ such that $L \subset {\rm Fix}(f^{k})$.	
\end{lem}

\begin{proof}
Take a point $x\in L-\{a\}$. Since $x\in R(f)$,  it is an arcwise connected limit point of the point sequence $\{f^n(x):n\in \mathbb{Z}_+\}$  by Theorem \ref{5.4}. Thus there exists some $k\in \mathbb{N}$ such that $f^k(x)\in L$.
\smallskip

We claim that $f^k(L)\subset L$. Otherwise,  there is some $y\in L$ with $f^k(y)\notin L$. Then by the connectivity of $f^k(L[x,y])$, we have that $a\in f^k(L[x,y])$. Then there is  $z\in L[x,y]$ with $f^k(z)=a$. Since $a\in {\rm Fix}(f)$, it follows that $z\notin R(f)$, which is a contradiction.
\smallskip

Now we prove that $L\subset {\rm Fix}(f^k)$. Assume to the contrary that there exists some point $x\in L$ such that $x\notin {\rm Fix}(f^k)$. Then we discuss into two cases.

\smallskip
\noindent {\bf Case 1.}\quad $L\cap {\rm Fix}(f^k)=\{a\}$. For any $x\in L-\{a\}$, since $f^k(L)\subset L$ by the above claim, we see that $O(x,f^k)\subset L$. Moreover, we have $O(x,f^k)\subset L-\{a\}$. (Otherwise, if  $f^{nk}(x)=a$ for some $n\in \mathbb{N}$, then $x\notin R(f)$, which is a contradiction.) This implies that either ${\rm lim}_{n\to \infty}f^{nk}(x)=a$, or the point sequence $\{f^{nk}(x):n\in \mathbb{N}\}$ is confinal with $L$, both of which imply that $x\notin R(f^k)$. This contradicts with the fact that $x\in R(f)$ by Lemma~\ref{recurrent}. So this case does not occur.

\smallskip
\noindent{\bf Case 2.}\quad There is a subarc $A=L[b,c]$ in $L$ such that $A\cap {\rm Fix}(f^k)=\{b,c\}$. (The point $b$ may be $a$.)
In this case, since $f^k(A)$ is connected and $L\cap {\rm Br}(X)=\emptyset$,  we have that $f^k({\AA})\subset {\AA}$. (Otherwise, there exists a point $x\in {\AA}$ with $f^k(x)\in \{b,c\}$. Then $x\notin R(f^k)=R(f)$, which is a contradiction.) It follows that ${\rm lim}_{n\to \infty}f^{nk}(x)=b$ or $c$ for all $x\in {\AA}$, which contradicts
the recurrence of $x$. So, Case 2 cannot occur.
\smallskip

All together, we complete the proof.
\end{proof}

%\begin{thm} \label{main thm}
%	Let $X$ be a quasi-graph, and $f:X \rightarrow X$ be a continuous map. Then $f$ is pointwise-recurrent if and only if one of the following two statements %holds:
%	\begin{itemize}
%		\item[(1)] $X$ is a circle and $f$ is a homeomorphism topologically conjugate to an irrational ratation of the unit circle $S^1$;
%		\item[(2)] $f$ is a perodic homeomorphism, i.e., there exists  $m \in \mathbb{N}$ such that $f^m=id_X$.
%	\end{itemize}
	
%\end{thm}

\begin{proof}[Proof of Theorem \ref{main thm}]
	If $X$ is a graph, then the theorem holds by Theorem~\ref{1.1}. So we suppose that $X$ contains at least one oscillatory quasi-arc. Let $X=G \cup (\bigcup_{i=1}^{n}L_i)$, where $G$ is a graph and $L_1\cdots,L_n$ are pairwise disjoint oscillatory quasi-arcs satisfying the conditions (1)-(4) of Theorem~\ref{2.24}. Furthermore, by Lemma~\ref{invariant graph}, we can assume that $G$ is  $f^k$-invariant for some $k\in \mathbb{N}$. Denote $g=f^k$. Then $g$ is  pointwise recurrent by Lemma~\ref{recurrent}. Specially, $g|_G$ is a pointwise recurrent on $G$. According to Theorem \ref{1.1}, there are two cases.

\smallskip
\noindent{\bf Case 1.}\quad $G$ is a circle and $g|_G$ is a homeomorphism topologically conjugate to an irrational rotation of the unit circle $\mathbb{S}^1$.	
However, this contradicts the assumption that $X$ contains at least one oscillatory quasi-arc. So, this case does not occur.

\smallskip
\noindent{\bf Case 2.}\quad $g|_G$ is a perodic homeomorphism of $G$.
 Then there exists some $p \in \mathbb{N}$ such $g^p|_G={\rm id}_G$. Denote $h=g^p$.  Since $a_i\in G\subset {\rm Fix}(h)$ and $a_i$ is the endpoint of $L_i$ for each $i\in \{1,\cdots,n\}$,  by Lemma~\ref{fix quasi-arc},
there exists $k_i\in \mathbb{N}$ such that $L_i \subset {\rm Fix}(h^{k_i})$. Take $m=kpk_1k_2\cdots k_n$, then $X=G \cup (\bigcup_{i=1}^{n}L_i)\subset {\rm Fix}(f^m)$. Hence $f$ is a periodic homeomorphism.
\end{proof}

\end{document}